\documentclass[12pt, reqno]{amsart}

\usepackage{amsfonts,amssymb,latexsym,amsmath, amsxtra,cancel,comment  }
\usepackage[utf8]{inputenc}
\usepackage{verbatim}
\usepackage{appendix}
\usepackage{marginnote}
\usepackage{enumitem}
\usepackage[OT2,T1]{fontenc}
\DeclareSymbolFont{cyrletters}{OT2}{wncyr}{m}{n}
\DeclareMathSymbol{\Sha}{\mathalpha}{cyrletters}{"58}

\theoremstyle{plain}
\newtheorem{theorem}{Theorem}[section]
\newtheorem*{theorem*}{Theorem}
\newtheorem{corollary}[theorem]{Corollary}

\newtheorem{proposition}[theorem]{Proposition}
\newtheorem*{conjecture*}{Conjecture}

\theoremstyle{definition}

\theoremstyle{remark}
\newtheorem*{remark}{Remark}
\newtheorem*{remarks}{Remarks}

\numberwithin{equation}{section}

\newenvironment{psmallmatrix}
{\left(\begin{smallmatrix}}
	{\end{smallmatrix}\right)}

\newcommand{\Z}{\mathbb{Z}}
\newcommand{\N}{\mathbb{N}}

\newcommand{\R}{\mathbb{R}}
\newcommand{\C}{\mathbb{C}}

\newcommand{\im}{\operatorname{Im}}
\renewcommand{\Im}{\operatorname{Im}}
\newcommand{\sgn}{\operatorname{sgn}}

\newcommand{\SL}{\text{\rm SL}}
\makeatletter
\newcommand*{\rom}[1]{\expandafter\@slowromancap\romannumeral #1@}
\makeatother

\DeclareMathOperator{\coeff}{coeff}

\makeatletter
\newcommand{\vast}{\bBigg@{4}}
\newcommand{\Vast}{\bBigg@{5}}
\makeatother

\renewenvironment{proof}[1][Proof]{\begin{trivlist} \item[\hskip \labelsep {\bfseries #1:}]}{\qed\end{trivlist}}

\title{Nontrivial Short}

\author{Kathrin Bringmann, Karl Mahlburg, Antun Milas}

\address{Mathematical Institute, University of Cologne, Weyertal 86-90, 50931 Cologne, Germany}
\email{kbringma@math.uni-koeln.de}

\address{Department of Mathematics, Louisiana State University, Baton Rouge, LA 70803, USA}
\email{mahlburg@math.lsu.edu}

\address{Department of Mathematics and Statistics, SUNY-Albany, Albany, NY 12222, U.S.A.}
\email{amilas@albany.edu}

\begin{document}

\title{On characters of $L_{\frak{sl}_n}(-\Lambda_0)$-modules}

\begin{abstract}
We use recent results of Rolen, Zwegers, and the first author to study characters of irreducible (highest weight) modules for the vertex operator
algebra $L_{\frak{sl}_\ell}(-\Lambda_0)$.  We establish asymptotic behaviors of characters for the (ordinary) irreducible
$L_{\frak{sl}_\ell}(-\Lambda_0)$-modules. As a consequence we prove that their quantum dimensions are one, as predicted by representation theory.
We also establish a full asymptotic expansion of irreducible characters for $\frak{sl}_3$. Finally, we determine a decomposition formula for the full characters in terms of unary theta and false theta functions which allows us to study their modular properties.
\end{abstract}
\maketitle

\section{Introduction and statement of results}

Vertex operator algebras have had a profound influence on mathematics and modern theoretical physics.
Arguably, the most important examples of vertex algebras are those associated to representations of the affine Kac-Moody Lie algebras at positive integral levels. Every such vertex algebra
is rational \cite{FZ} and its characters are modular functions on the same congruence subgroup . Moreover, characters of representations of affine vertex algebras at certain rational levels, called  admissible, are also modular  \cite{KW1}.

\subsection{Affine vertex algebras at negative levels}
More recently, a notable effort has been made to understand modules of
affine vertex algebras at a negative integer level \cite{AP,AP2,AK,AMo,KW2}, etc. Although there is no general Weyl-Kac type formula at these levels, in a few cases their characters (or  $q$-characters) are known explicitly. Here a prominent role is played by the so-called Deligne series of representations at the level $-\frac{h^\vee}{6}-1$ for the Lie algebras of type $D_4,E_6,E_7,$ and $E_8$ \cite{AK,KW2}. These vertex algebras are irrational and quasi-lisse, and thus their characters are quasi-modular and are also solutions of certain modular linear differential equations  \cite{AK}. The search for new examples off quasi-lisse (and lisse) vertex algebras is a very active area of current research with important applications to $N=2$ four dimensional CFT \cite{BLL} and logarithmic CFT \cite{CM2}.

Apart from Deligne's series, specialized characters of $L_{\frak{sl}_n}(-\Lambda_0)$-modules have been studied in several recent works
\cite{BCR,KW2}. This vertex algebra is no longer quasi-lisse but their characters still enjoy interesting ``quantum'' modular properties.  More precisely, it was proven in \cite{BCR,BRZ} that the
irreducible characters of level $-1$ are mixed quantum modular forms by realizing them as Fourier coefficients of meromorphic Jacobi forms of negative index.

\subsection{Quantum dimensions and asymptotics} Somewhat independent of these developments, in \cite{CM2, CMW} the authors initiated a study of asymptotic
expansions of characters of modules for irrational vertex algebras. Motivated by the success of the theory of rational vertex algebras \cite{H}, it is expected
that asymptotic expansions of characters can be used to formulate a Verlinde-type formula for the fusion rules.
For the singlet vertex algebra and generalizations this was recently established by using the so-called regularized quantum dimension \cite{CM2}.
Despite the progress made it is still unclear how to formulate a Verlinde-type formula for more general vertex algebras. The main issue is that vertex algebras can have representations
with very different properties, which makes it difficult to define quantum dimensions, $S$-matrices, and other standard invariants.
But for vertex algebras whose irreducible modules are $C_1$-cofinite we expect that asymptotically (as $t \to 0^+$)
\begin{equation} \label{quantum-asym}
\frac{{\rm ch}_M(it)}{{\rm ch}_V(it)} \sim a_0 +a_1 t+ \cdots
\end{equation}
for any irreducible $V$-module $M$. In particular, \eqref{quantum-asym} would imply $\lim_{t \to 0^+}\frac{{\rm ch}_M(it)}{{\rm ch}_V(it)}=a_0$, so we have a well-defined notion of quantum dimension (the existence of the limit does not necessarily require (\ref{quantum-asym}) though).
In order to gain a better insight into irrational vertex algebras it is desirable to study asymptotic expansion of characters, and hence of (\ref{quantum-asym}), for various families of affine vertex algebras (for related $W$-algebra computations see \cite{CM2}).

\subsection{Main results} In this paper we study asymptotic and modular properties of characters of the vertex algebra $L_{\frak{sl}_n} (-\Lambda_0)$-modules, $n \geq 3$. This vertex algebra is a good representative of an irrational affine vertex algebra admitting
both atypical and typical modules. We only consider atypical representations, leaving Verlinde-type formula for future considerations.

Let us outline the main results.
For brevity, we let $\frak{g}=\frak{sl}_\ell$, $\ell \geq 3$, the Lie algebra of trace zero $\ell \times \ell$ matrices. As usual, from $\frak{g}$ we construct the affine Kac-Moody Lie
algebra $\widehat{\frak{g}}$. Denote by $\omega_j$, $j \in \{1,...,\ell \}$, the fundamental weights of $\frak{g}$ and by $\Lambda_j$, $j \in \{0,...,\ell-1\}$, the fundamental weights for $\widehat{\frak{g}}$. For each weight $\Lambda$,  let $L_{\frak{g}}(\Lambda)$ be the corresponding irreducible highest weight $\widehat{\frak{g}}$-module.

We equip $L_{\frak{g}}(-\Lambda_0)$ with a conformal vertex  algebra structure via  Sugawara's construction.
It is known by Adamovic and Perse \cite{AP} that  $L_{\frak{g}}(-\Lambda_0)$ admits countably many inequivalent ordinary modules: $L_{\frak{g}}(- (s+1) \Lambda_0  +s \Lambda_1)$ with $s \in \mathbb{N}_0$ and $L_{\frak{g}}(- (s+1) \Lambda_0  + s  \Lambda_{\ell-1})$ with $s \in \mathbb{N}$, and infinitely many generic modules \cite{AP2}.
We denote the degree operator by $L(0)$ and the central charge by $c:=\frac{{\rm dim}(\frak{g}) (-1)}{(-1)+h^\vee}=-(\ell+1)$, where $h^\vee=\ell$ is the dual Coxeter number of $\frak{sl}_\ell$. We use this data to define the relevant characters (here $\Lambda \in
\{ \Lambda_1, \Lambda_{\ell-1} \}$ and $q:=e^{2\pi i\tau}$):
$${\rm ch}[ L_{\frak{g}}(-(s+1)\Lambda_0  +s \Lambda)  ](\tau):={\rm tr}_{L_{\frak{g}}(- (s+1) \Lambda_0  +s \Lambda)} q^{L(0)-\frac{c}{24}}.$$
Let 
$$h_{s}:=\frac{(s \omega_1+2 \rho,s \omega_1)}{2((-1)+h^\vee)}=\frac{(s \omega_{\ell-1}+2 \rho,s \omega_{\ell-1})}{2((-1)+h^\vee)}=\frac{s^2}{2 \ell}+\frac{s}{2}$$ denote the lowest conformal weight of $L_{\frak{g}}(- (s+1) \Lambda_0  +s \Lambda_1)$ and  of $L_{\frak{g}}(- (s+1) \Lambda_0  +s \Lambda_{\ell-1})$. Note that $L_{\frak{g}}(-\Lambda_0 (s+1) +s \Lambda_1)$ and $L_{\frak{g}}(- (s+1) \Lambda_0   + s  \Lambda_{\ell-1})$ are dual to each other so they share the same
character.


By using results of Kac and Wakimoto on specialized characters of the affine Lie algebras $\widehat{\frak{sl}}_\ell$ at level $-1$ \cite[Section 1]{KW2}, we obtain, for $s \in\N_0$ and
$\Lambda \in \{ \Lambda_1,\Lambda_{\ell-1} \}$,
\begin{equation}\label{constantcoeff}
{\rm ch}\big[L_{\frak{g}}(- (s+1) \Lambda_0 +s \Lambda)\big](\tau) = q^{h_{s}-\frac{c}{24}} \frac{F_{\ell, s}(q)}{(q)_\infty^{\ell^2-1}},
\end{equation}
where
\begin{equation*}
F_{\ell,s}(q) := (q)_\infty^{\ell^2-1} \coeff_{\zeta^s} \frac{(q)_\infty}{(\zeta)_\infty^\ell\left(\zeta^{-1}q\right)_\infty^\ell}.
\end{equation*}
Here $(a)_{\infty}=(a;q)_{\infty}:=\prod_{j=0}^\infty (1-aq^j)$. Throughout, we are interested in the range $0<y<v$, where we write $z=x+iy$, $\tau=u+iv$. We note that other ranges can be treated in a similar way.

\begin{theorem}\label{main}
	We have, as $t\to 0^+$,
	\begin{equation*}
	F_{\ell,s}\left(e^{-t}\right) \sim \mathcal C_\ell \left(\frac{t}{2 \pi}\right)^{1-\frac{\ell^2}{2}}e^{-\frac{\pi^2\left(\ell^2-2\ell\right)}{6t}},
	\end{equation*}
	where
	\begin{equation*}
	\mathcal C_\ell := \frac{2^{1-2\ell}(\ell-1)!}{\Gamma\left(\frac{\ell+1}{2}\right)^2}.
	\end{equation*}
\end{theorem}

The following corollary follows directly (the same results apply for their dual modules, $L_{\frak{g}}(- (s+1) \Lambda_0  +s \Lambda_{\ell-1})$ as their $q$-characters are equal).
Let $V$ be a vertex algebra. For a $V$-module $M$ we define its {\em quantum} or {\em asymptotic dimension}:
$${\rm qdim}(M):=\lim_{t \to 0^+} \frac{{\rm ch}[M](it)}{{\rm ch}[V](it)}.$$
\begin{corollary}\label{cor:intro} Let $\frak{g}=\frak{sl}_{\ell}$, $\ell \geq 3$ and $s \in \mathbb{N}_0$. Let $V_s$ denote $L_{\frak{g}}(-(s+1)\Lambda_0 + s \Lambda_1)$ or $L_{\frak{g}}(-(s+1)\Lambda_0 + s \Lambda_{\ell-1})$.
Then we have, as $t \to 0^+$,
\begin{equation}\label{asVs}
{\rm ch}\big[V_s \big]\left(e^{-t}\right) \sim \mathcal C_\ell \sqrt{\frac{t}{2\pi}} e^{\frac{\pi^2\left(2 \ell-1 \right)}{6t}}.
\end{equation}
In particular ${\rm qdim}(V_s)=1$ for every $\ell$ and $s$.
\end{corollary}
In fact, the method used to prove Theorem \ref{main} should in principle lead to the full asymptotic expansion for any such character. In Proposition \ref{P:ell=3Asymp} below, we carry out the details for the case $\ell = 3$ below.


Now we recall Kac and Wakimoto's work in \cite[Section 1]{KW2}, which gives the full character as the coefficients of certain series. Let $\zeta_j :=e^{\alpha_j}$ and define
\begin{equation*}
F_{\ell}(\zeta_1,...,\zeta_\ell):=(q)_\infty \prod_{j=1}^\ell \prod_{k=1}^\infty \frac{1}{\left(1-\zeta_j^{-1} \cdots \zeta^{-1}_{\ell} q^{k}\right)\left(1-\zeta_j \cdots \zeta_{\ell} q^{k-1}\right)}.
\end{equation*}
Using the geometric series on every term, we expand this product as a formal series, to obtain a decomposition of the shape
\begin{equation*}
F_{\ell}(\zeta_1,...,\zeta_\ell) =: \sum_{s \in \mathbb{Z}} F_{\ell,s}(\zeta_1,...,\zeta_{\ell-1}) \zeta_\ell^s.
\end{equation*}
In other words, $F_{\ell,s}(\zeta_1,...,\zeta_{\ell-1})$ is defined to be the $s$-th Fourier coefficient of $F_{\ell}$ with respect to $\zeta_{\ell}$.
In \cite{KW2}, Kac and Wakimoto showed that the $F_{\ell,s}$ are essentially the irreducible characters of $L_{\frak{g}}(-\Lambda_0)$-modules.
More precisely $(\zeta_j:=e^{2\pi iz_j})$
\begin{multline}
\label{E:F=fullch}
F_{\ell}(\zeta_1,...,\zeta_\ell)=\sum_{s=1}^\infty (\zeta_1 \cdots \zeta_\ell)^{s} q^{\frac{3s}{2}}  \widetilde{\rm ch}[-(1+s)\Lambda_0+s \Lambda_1](z;\tau)\\
+\sum_{s=0}^\infty \zeta_\ell^{-s} q^{-\frac{s}{2}} \widetilde{\rm ch}[-(1+s)\Lambda_0+s \Lambda_{\ell-1}](z;\tau ),
\end{multline}
where $\widetilde{{\rm ch}}[M](z;\tau)$ is the full character of $M$  and $\zeta=(\zeta_1,...,\zeta_{\ell-1})$.
Note that our definition  differs from \cite{KW2} due to the shift $\zeta_\ell \mapsto \zeta_\ell q^{\frac12}$, which affects the $q$-powers in \eqref{E:F=fullch}.

In Section 7 we prove the following decomposition result for $F_{\ell,s}(\zeta_1,\ldots,\zeta_{\ell-1})$.
\begin{theorem} \label{T:Fls}
Assume that $w_j:=-\sum_{k=j}^{\ell-1} z_k$ for $1 \leq j \leq \ell-1$, and $w_{\ell}=0$ are all distinct.
We have, for $|q|<|\zeta_j|^{\ell}<1$, $1\leq j\leq\ell-1$,
\begin{align*}
&F_{\ell,s}(\zeta_1,\ldots,\zeta_{\ell-1})= -i^{\ell+1}  q^{-h_s + \frac{c}{24}}\eta(\tau)^{\ell - 2} \prod_{j=1}^{\ell-1} \zeta_j^{\frac{js}{\ell}}  \sum_{\nu=1}^{\ell}\frac{\vartheta_{s-\frac{\ell}{2},\ell,\frac{\ell}{2}}^+\left(w_\nu-\frac{1}{\ell}\sum_{j=1}^{\ell}w_j \right)}
{\prod_{\substack{j=1\\j\neq\nu}}^{\ell} \vartheta\left(w_\nu-w_j\right)}.
	\end{align*}
where $\vartheta$ and $\vartheta_{s-\frac{\ell}{2},\ell,\frac{\ell}{2}}^+$ are certain theta and partial theta functions defined in \eqref{theta} and \eqref{ptheta}, respectively.
\end{theorem}
Kac and Wakimoto \cite{KW2} recently obtained a Weyl-Kac type formula for the character ${\rm ch}[-(1+s)\Lambda_0+s \Lambda](z;\tau)$. Unlike our
Theorem \ref{T:Fls}, where the whole character splits into rank one pieces, their formula involves a summation over over the half-space of the root lattice of $\frak{sl}_{\ell}$ (see also \cite{BCR}). Although very elegant, we found their result more difficult to use for the purpose of computing asymptotic expansions.

\subsection{Organization  of the paper}
The paper is organized as follows. In Section \ref{sec:pre} we recall basic facts about Bernoulli numbers and their generalizations, theta functions, and the Euler-Maclaurin summation formula. We then study in Section \ref{sec:asympaux} the asymptotic behaviour of certain auxiliary functions. Section \ref{sec:qseries} decomposes $F_{\ell,s}$ in terms of partial theta functions and quasimodular forms. In Section \ref{sec:mainproof}, we then finish the proof of Theorem \ref{main} and Corollary \ref{cor:intro}. Section \ref{sec:caseell3} treats the case $\ell=4$ in which we find a full asymptotic expansion for the corresponding characters. We prove a decomposition formula of $F_{\ell,s}(\zeta_1,\ldots,\zeta_{\ell-1})$ in Section \ref{sec:decomp}. We finish the paper in Section 8 by computing modular transformation properties of the full character $F_{\ell,s}$ from modular properties of unary partial and false theta functions, which we also derive there.

\section*{Acknowledgments}
We thank Victor Kac for raising the question of computing asymptotic expansion of characters.

The research of the  first author is supported by the Alfried Krupp Prize for Young University Teachers of the Krupp foundation and the research leading to these results receives funding from the European Research Council under the European Union's Seventh Framework Programme (FP/2007-2013) / ERC Grant agreement n. 335220 - AQSER. The research of the third author was supported by a Simons Foundation Collaboration Grant ($\#$ 317908) and NSF-DMS Grant 1601070.

\section{Preliminaries}\label{sec:pre}
\subsection{Bernoulli number generating functions} Recall that the {\it Bernoulli numbers} $B_k$ are defined via their generating function
\begin{equation}\label{Bernoulli}
\frac{w}{e^w-1}=:\sum_{k=0}^{\infty}B_k\frac{w^k}{k!}.\qquad(|w|<2\pi)
\end{equation}
We also require the following modified generating function, again assuming $|w|<2\pi$,
\begin{equation}\label{S}
S(w) := \sum_{k = 1}^\infty \frac{B_{2k} w^{2k}}{2k (2k)!}.
\end{equation}
Differentiating \eqref{S} and using \eqref{Bernoulli} gives that
	\begin{equation*}
	S'(w)= \frac{e^{-w}}{1 - e^{-w}} - \frac{1}{w} + \frac12
	\end{equation*}
and hence, by taking the limit on both sides,
\begin{equation}\label{Bernid}
	S(w) = \operatorname{Log}\left(\frac{e^{w}-1}{w}\right) - \frac{w}{2}.
\end{equation}

	We also need generalizations of Bernoulli numbers. For this, denote by $B_n^{(r)}(x)$ the {\it Bernoulli polynomials of order} $r\in\N$, defined via their generating function
	\begin{equation*}\label{Berpol}
	\left(\frac{w}{e^w-1}\right)^r e^{xw} =: \sum_{n=0}^\infty B_n^{(r)}(x) \frac{w^n}{n!}.
	\end{equation*}
For $r=1$ this recovers the ordinary Bernoulli polynomials $B_n(x)$.

We also require the \emph{Euler polynomials} defined through the generating function
\begin{equation*}
\frac{2e^{xw}}{e^w+1} =: \sum_{n=0}^{\infty} \frac{E_n(x)}{n!}w^n.
\end{equation*}
The \emph{Euler numbers} are in particular given by
\begin{equation}\label{Euler12}
E_n = 2^n E_n\left(\frac12\right).
\end{equation}
The following identity relates Euler and Bernoulli polynomials ($n\in\N, m\in 2\N$)
\begin{equation}\label{EB}
E_n(mx) = -\frac{2}{n+1} m^n \sum_{k=0}^{m-1} (-1)^k B_{n+1}\left(x+\frac{k}{m}\right).
\end{equation}

\subsection{Theta functions}
Let
\begin{equation}
\label{theta}
\vartheta(z)=\vartheta(z;\tau):=\sum_{n\in\frac12+\mathbb Z}q^{\frac{n^2}{2}}e^{2\pi in\left(z+\frac12\right)},\  \eta(\tau):=q^{\frac{1}{24}}(q)_{\infty}.
\end{equation}
By the Jacobi triple product identity, we have the product expansion
\begin{equation*}
\vartheta(z;\tau)=-iq^{\frac{1}{8}}\zeta^{-\frac12}(q)_{\infty}(\zeta)_{\infty}(\zeta^{-1}q)_{\infty}.
\end{equation*}
Recall the transformations ($\lambda,\mu\in\Z$, $\left(\begin{smallmatrix}
a & b \\ c & d
\end{smallmatrix}\right)\in\SL_2(\Z)$)
\begin{align}\label{thetatran}
\vartheta(z+\lambda \tau + \mu) &= (-1)^{\lambda+\mu} q^{-\frac{\lambda^2}{2}} \zeta^{-\lambda} \vartheta(z),\\
\vartheta\left(\frac{z}{c\tau+d};\frac{a\tau+b}{c\tau+d}\right)&=\chi\left(\begin{matrix}
a & b \\ c & d
\end{matrix}\right) (c\tau+d)^{\frac12} e^{\frac{\pi i cz^2}{c\tau+d}}\vartheta(z;\tau),\notag
\end{align}
where $\chi$ is the multiplier of $\eta^3$. In particular
\begin{align}
\vartheta\left(\frac{z}{\tau};-\frac{1}{\tau}\right)&=-i\sqrt{-i\tau}e^{\frac{\pi iz^2}{\tau}}\vartheta(z;\tau),\qquad
\eta\left(-\frac{1}{\tau}\right)=\sqrt{-i\tau}\eta(\tau).\label{etatran}
\end{align}

Next recall that
\begin{equation}\label{thetatay}
\vartheta(z;\tau)=-2\pi z\cdot\eta(\tau)^3\cdot\exp\left(-\sum_{k=1}^\infty \frac{G_{2k}(\tau)}{2k}z^{2k}\right),
\end{equation}
where for $k\in\N$
\begin{equation}
\label{E:G2k}
G_{2k}(\tau):= -\frac{ (2\pi i)^{2k} B_{2k}}{(2k)!} + \frac{2 (2\pi i)^{2k}}{(2k-1)!} \sum_{n = 1}^\infty \sigma_{2k - 1}(n) q^n.
\end{equation}
Here $\sigma_\ell(n):=\sum_{d\mid n}d^\ell$ is the {\it $\ell$-th divisor sum}.
We also require the transformations $(k>1)$
\begin{align}\label{Eistran}
G_{2k}(\tau)=\tau^{-2k}G_{2k}\left(-\frac{1}{\tau}\right)
\qquad G_2(\tau)=\tau^{-2}G_{2}\left(-\frac{1}{\tau}\right) +\frac{2\pi i}{\tau}.
\end{align}
\subsection{Euler-Maclaurin summation formula}
The Euler-Maclaurin summation formula (see e.g. \cite{Za}, correcting a sign error) implies that for $\alpha\in\R$ and $f:\R\to \R$ a $C^\infty$-function such that $f$ and all its derivatives are of rapid decay at infinity, we have
\begin{equation}\label{EM}
\sum_{n= 0}^\infty f((n+\alpha)t) = \frac{I_f}{t} - \sum_{n= 0}^N \frac{B_{n+1}(\alpha)}{(n+1)!} f^{(n)}(0)t^n+O\left(t^{N+1}\right)\quad (t\to 0^+),
\end{equation}
where $B_\ell(x)$ denotes the $\ell$-th Bernoulli polynomial and $I_f:=\int_0^\infty f(x)dx$.

\section{Asymptotics for auxilary functions}\label{sec:asympaux}
We next determine the asymptotic behaviour of two partial theta functions. These expansions are well-known to experts; however, for the convenience of the reader, we provide proofs.

The first function we require is (here $j\in\N_0$, $r\in\mathbb Q$),
\begin{equation}\label{defineF}
\mathcal{F}_{j,r}(t):=2^{-2j}t^j\sum_{n= 0}^\infty(-1)^n \left(n+r\right)^{2j} e^{-\frac14\left(n+r\right)^2 t}.
\end{equation}
Then
\begin{align*}
\mathcal{F}_{j,r}(t)=\sum_{n= 0}^\infty\left(\mathbb F_j\left(\left(n+\frac{r}{2}\right)\sqrt{t}\right)-\mathbb F_j\left(\left(n+\frac{r+1}{2}\right)\sqrt{t}\right)\right),
\end{align*}
where
$
\mathbb F_j(x):=x^{2j}e^{-x^2}.
$
By \eqref{EM}, we obtain
\begin{equation}\label{fullF}
\mathcal{F}_{j,r}(t)=-\sum_{n= 0}^N\frac{B_{n+1}\left(\frac{r}{2}\right)-B_{n+1}\left(\frac{r+1}{2}\right)}{(n+1)!}\mathbb F_j^{(n)}(0)t^{\frac{n}{2}}+O\left(t^{\frac{N+1}{2}}\right).
\end{equation}
Noting that $\mathbb F_j$ is an even function, we compute
\begin{align}
\mathcal F_{j,r}(t)=-\sum_{n=0}^N \frac{B_{2n+2j+1}\left(\frac{r}{2}\right) - B_{2n+2j+1}\left(\frac{r+1}{2}\right)}{(2n+2j+1)n!} (-1)^n t^{n+j}+O\left(t^{N+j+1}\right).\label{fullF2}
\end{align}
This implies that
\begin{align}\label{Fas}
\mathcal F_{j,r}(t) \sim -\frac{B_{2j+1}\left(\frac{r}{2}\right)- B_{2j+1}\left(\frac{r+1}{2}\right)}{2j+1}  t^j.
\end{align}
In particular,
\begin{equation}\label{Fas2}
\mathcal{F}_{0,r}(t) \sim \frac 12.
\end{equation}

We next study, for $j \in \mathbb{N}$, $r\in\mathbb Q$,
\begin{equation*}
\mathcal G_{j,r}(t) := t^{j-\frac12} \sum_{n = 0}^\infty \left(n+r\right)^{2j-1} e^{-(n+r)^2t}.
\end{equation*}
Then
\begin{equation*}
\mathcal G_{j,r}(t) = \sum_{n= 0}^\infty \mathbb G_j\left(\left(n+r\right)\sqrt{t}\right),
\end{equation*}
where  $\mathbb G_j(x):=x^{2j-1}e^{-x^2}$. By \eqref{EM} and the fact that $I_{\mathbb G_j} = \frac{(j-1)!}{2}$, we obtain that
\begin{equation}\label{Gas}
\mathcal G_{j,r}(t) = \frac{I_{\mathbb G_j}}{\sqrt{t}} - \sum_{n= 0}^N \frac{B_{n+1}\left(r\right)}{(n+1)!} \mathbb G_j^{(n)}(0) t^{\frac{n}{2}}+O\left(t^{\frac{N+1}{2}}\right) \sim \frac{(j-1)!}{2\sqrt{t}}.
\end{equation}

\section{ A $q$-series for $F_{\ell,s}$}\label{sec:qseries}
To prove Theorem \ref{main}, we first find a $q$-series representation for $F_{\ell,s}$. For this, we use a theorem from \cite{BRZ}. To state it, we let, for $z_0\in\C$,
	\[
	P_{z_0}:=z_0+[0, 1)+\tau[0, 1).
	\]
		Let $\phi$  be a meromorphic function  satisfying ($\lambda, \mu\in\Z, \, \varepsilon \in \{0, 1\}$, $m\in-\frac12\N$)
		\begin{equation}\label{elliptic}
		\phi(z+\lambda \tau+\mu)=(-1)^{2m \mu+\lambda\varepsilon}
		e^{-2\pi im\left(\lambda^2\tau+2\lambda z\right)}\phi(z).
		\end{equation}
	Moreover $S_{z_0}(\tau)$ is the complete  set of poles of $\phi$ in $P_{z_0}$. 	 Furthermore $\mathcal{D}_z:=\frac1{2\pi i}\frac{\partial}{\partial z}$, and for $M\in\frac12\N$, $r \in M+\Z, \, \varepsilon \in \{0,1\}$, define the partial/false theta functions
	\begin{equation}
	\label{ptheta}
	\vartheta_{r,\varepsilon,M}^+(z)=\vartheta_{r, \varepsilon, M}^+(z; \tau):=\sum_{n= 0}^\infty(-1)^{n \varepsilon}\zeta^{2Mn-r} q^{\frac1{4M}(2Mn-r)^2}.
	\end{equation}
For $r\in m+\mathbb Z$, let $h_{r,z_0}(\tau)$ be the renormalized $r$-th Fourier coefficient with respect to $z_0$
\begin{equation}
\label{E:hr}
h_r(\tau) = h_{r,z_0}(\tau) := q^{-\frac{r^2}{4m}} \int_{z_0}^{z_0+1} \phi(z) e^{-2\pi ir z} dz.
\end{equation}	
If $\phi$ has any poles on the line from $z_0$ to $z_0 + 1$, then the contour must be deformed. In particular, if there is a pole that is not an endpoint, then we define the path to be the average of the paths deformed to pass above and below the pole. If there is a pole at an endpoint, we replace the path by $[z_0-\delta, z_0+1-\delta]$ with $\delta$ small, so that there is no pole at an endpoint.
We now recall the following Theorem from \cite{BRZ} (slightly rewritten).
\begin{theorem}\label{hltheorem}
 Assume that $\phi$  satisfies \eqref{elliptic} and choose $z_0\in\C$, such that $\phi$ has no poles on the boundary of $P_{z_0}$.
	Then we have, for any $r\in m+\Z$,
\begin{equation}\label{reph}
h_{r,z_0}(\tau) = 2\pi i\sum_{z_s\in S_{z_0}(\tau)} \underset{\ z=z_s}{\rm Res}\left( \phi(z,\tau) \vartheta_{r,\varepsilon,-m}^+(z;\tau)\right).
\end{equation}
\end{theorem}
We use Theorem \ref{hltheorem} to find a $q$-series for $F_{\ell,s}$.
\begin{theorem}
\label{T:F=Dj}
	We have, for $0<y<v$,
	\begin{equation*}
	F_{\ell,s}(q) = i^\ell (q)^{\ell^2-2\ell}_{\infty} q^{-h_s-\frac{\ell}{8}}\sum_{\substack{1\leq j \leq \ell \\ j\equiv \ell\pmod{2}}} \frac{D_{-j}(\tau)}{(j-1)!}\sum_{n= 0}^\infty(-1)^{n\ell}\left(\ell n+\frac{\ell}{2}-s\right)^{j-1}q^{\frac{1}{2\ell}\left(\ell n+\frac{\ell}{2}-s\right)^2},
	\end{equation*}
	where the $D_{-j}(\tau)$ are the Laurent coefficients of $g_\ell$ in $z$ (expanded around $0$), so that
	\begin{equation}
	\label{E:gDdef}
	g_\ell(z;\tau) = \sum_{j = 1}^\ell \frac{D_{-j}(\tau)}{(2\pi i z)^j} + O(1),
	\end{equation}
\end{theorem}
\begin{proof}
To apply \eqref{reph}, we rewrite
$$
F_{\ell,s}(q) = (q)_\infty^{\ell^2 - 1} \coeff_{\zeta^s}\frac{(q)_\infty }{(\zeta)_\infty ^\ell \left(\zeta^{-1}q\right)_\infty ^\ell}=(q)_\infty^{\ell^2 - 2\ell} f_{\ell,s}(\tau)
$$
where
\begin{equation}\label{frog}
f_{\ell,s}(\tau):=(-i)^\ell \operatorname{coeff}_{\zeta^{s+\frac{\ell}{2}}} g_\ell(z;\tau),
\end{equation}
with
$$
g_\ell(z)=g_\ell(z;\tau):=\frac{\eta(\tau)^{3\ell}}{\vartheta(z;\tau)^\ell}.
$$
Note that $g_\ell$ is a Jacobi form of weight $\ell$ and index $m = -\frac{\ell}{2}$. In particular we have the elliptic transformation
($\lambda,\mu\in\mathbb Z$)
\begin{equation}\label{elliptictran}
g_\ell(z+\lambda\tau+\mu)=(-1)^{(\lambda +\mu)\ell}\zeta^{\lambda\ell}q^{\frac{\ell\lambda^2}{2}}g_\ell(z).
\end{equation}
Thus \eqref{elliptic} holds with $\varepsilon\equiv\ell\pmod{2}$.
Moreover, for $\begin{psmallmatrix}
	a & b\\c & d
\end{psmallmatrix}\in\text{SL}_2(\mathbb Z)$ we require the modular transformation
\begin{equation*}
g_{\ell}\left(\frac{z}{c\tau+d};\frac{a\tau+b}{c\tau+d}\right)=(c\tau+d)^{\ell}e^{-\frac{\pi i\ell cz^2}{c\tau+d}}g_\ell(z;\tau).
\end{equation*}

To rewrite \eqref{frog},
we use \eqref{reph}, to obtain
\[
H_{s+\frac{\ell}{2}}(\tau)=2\pi i \cdot \underset{z=\tau}{\text{Res}}\left(g_\ell(z)\vartheta^+_{s+\frac{\ell}{2},\varepsilon,\frac{\ell}{2}}(z;\tau)\right),
\]
where $H_j(\tau)$ denotes the $j$-th renormalized $h_{j,z_0}(\tau)$ (see \eqref{E:hr}) Fourier coefficient of $g_\ell(z;\tau)$ in the fixed range $0<y<v$, and thus $z_0$ is chosen such that it satisfies the conditions of Theorem \ref{hltheorem} with $\tau \in P_{z_0}$.
Using \eqref{elliptictran} with $\lambda=1$ and $\nu=0$ (and $z \mapsto z - \tau$), we directly obtain that
$$
g_\ell(z) =
 (-1)^\ell \zeta^{\ell} q^{-\frac{\ell}{2}} g_{\ell}(z-\tau).
$$
Plugging in, we find that
\begin{align*}
H_{s+\frac{\ell}{2}}(\tau) & = (-1)^{\ell}q^{-\frac{\ell}{2}}2\pi i  \underset{z=\tau}{\text{Res}}\left(g_\ell(z-\tau;\tau)\zeta^{\ell}\vartheta^+_{s+\frac{\ell}{2},\varepsilon,\frac{\ell}{2}}(z;\tau)\right) \\
& =(-1)^{\ell}q^{-\frac{\ell}{2}}\sum_{\substack{j=1 \\ j\equiv \ell\pmod{2}}}^{\ell}\frac{D_{-j}(\tau)}{(j-1)!}\left[\mathcal D_z^{j-1}\left(\zeta^{\ell}\vartheta^+_{s+\frac{\ell}{2},\varepsilon,\frac{\ell}{2}}(z;\tau)\right)\right]_{z=\tau}.
\end{align*}
Computing the derivative
\begin{align*}
&\left[\mathcal D_z^{j-1}\left(\zeta^{\ell}\vartheta^+_{s+\frac{\ell}{2},\varepsilon,\frac{\ell}{2}}(z;\tau)\right)\right]_{z=\tau}=q^{\frac{\ell}{2}}\sum_{n= 0}^\infty(-1)^{n\varepsilon}\left(\ell n+\frac{\ell}{2}-s\right)^{j-1}q^{\frac{1}{2\ell}\left(\ell n+\frac{\ell}{2}-s\right)^2}
\end{align*}
then gives
\begin{equation}\label{hform}
H_{s+\frac{\ell}{2}}(\tau)=(-1)^{\ell}\sum_{\substack{1\leq j\leq \ell \\ j\equiv \ell\pmod{2}}}\frac{D_{-j}(\tau)}{(j-1)!}\sum_{n= 0}^\infty(-1)^{n\varepsilon}\left(\ell n+\frac{\ell}{2}-s\right)^{j-1}q^{\frac{1}{2\ell}\left(\ell n+\frac{\ell}{2}-s\right)^2}.
\end{equation}
Noting that
\begin{align}
\label{E:F=h}
F_{\ell,s}(q) = (-i)^\ell (q)_\infty^{\ell^2 - 2\ell}  q^{-h_s-\frac{\ell}{8}} H_{s+\frac{\ell}{2}}(\tau)
\end{align}
finishes the claim.
\end{proof}

\section{Proof of Theorem \ref{main}}\label{sec:mainproof}
\subsection{Asymptotic series for quasimodular forms}

Using \eqref{etatran} and Corollary \ref{cor:intro} gives that
\begin{equation}\label{etaas}
\eta\left(\frac{it}{2\pi}\right)  \sim \sqrt{\frac{2\pi}{t}} e^{-\frac{\pi^2}{6t}}\quad \left(t\to 0^+\right).
\end{equation}
Plugging this and Corollary \ref{cor:intro}  into \eqref{E:F=h} we then obtain
\begin{equation}\label{asFh}
F_{\ell,s}\left(e^{-t}\right)\sim(-i)^{\ell}\left(\frac{2\pi}{t}\right)^{\frac{\ell^2}{2}-\ell}e^{-\frac{\pi^2}{12t}\left(\ell^2-2\ell\right)}H_{s+\frac{\ell}{2}}\left(\frac{it}{2\pi}\right).
\end{equation}
To determine the asymptotic behavior of $H_{s+\frac{\ell}{2}}$, we first investigate the coefficients in form of $D_{-j}$ in \eqref{hform}. Using \eqref{thetatay}, we write
\begin{equation}
\label{E:git2pi}
g_\ell\left(z;\frac{it}{2\pi}\right)=\frac{1}{(-2\pi z)^\ell}\exp\left(\ell\sum_{k=1}^{\infty}\frac{G_{2k}\left(\frac{it}{2\pi}\right)}{2k}z^{2k}\right).
\end{equation}
 From this we see that in $g_\ell(z; \tau)$ every coefficient in front of $z^{-j}$ has the shape
\begin{equation}
\label{E:Gktuples}
D_{-j}(\tau) = \sum_{m,k} c(m, k) G^{m_1}_{k_1}(\tau)\cdots G^{m_r}_{k_r}(\tau),
\end{equation}
where the summation is over all tuples of positive integers (of any finite length) $m = (m_1, \dots, m_r), k = (k_1, \dots, k_r)$ such that $m_1k_1+\dots +m_r k_r = \ell - j$ and the $c(m,k)$ are real constants depending on the tuple.
By \eqref{Eistran}, we have that, for $k \geq 2$,
\begin{equation*}
G_{2k}\left(\frac{it}{2\pi}\right) = \left(\frac{2\pi}{it}\right)^{2k} \left(-\frac{ (2\pi i)^{2k} B_{2k}}{(2k)!} + \frac{2 (2\pi i)^{2k}}{(2k-1)!} \sum_{n = 1}^{\infty} \sigma_{2k - 1}(n) e^{-\frac{4 \pi^2 n}{t}} \right).
\end{equation*}
Thus we obtain, for $k \geq 2$,
\begin{equation}
\label{E:G2kAsymp}
G_{2k}\left(\frac{it}{2\pi}\right) \sim -\frac{\left(\frac{4 \pi^2}{t}\right)^{2k} B_{2k}}{(2k)!},
\end{equation}
with an error term that decays exponentially as a function of $\frac{1}{t}$. For $G_2(\frac{it}{2\pi})$ the main term is also given by \eqref{E:G2kAsymp}, but now the error is only smaller by the polynomial factor $t$ due to the additional term $\frac{4\pi^2}{t}$ in the transformation \eqref{Eistran}. Regardless, if we now plug \eqref{E:G2kAsymp} into \eqref{E:git2pi}, the accumulated error term in \eqref{E:Gktuples} is smaller than the main term by at least a polynomial factor of $t$.

Furthermore, note that it is valid to calculate any given Laurent coefficient in \eqref{E:git2pi} by restricting $z$ to a segment of the imaginary axis, so $z = iy$ with $0 < y < v = \frac{t}{2\pi}$. This is a subset of the canonical analytic domain $0 < y < v.$ This restriction is convenient as it allows us to package all of the main asymptotic terms into a generating function, which can then be evaluated using the identities from Section \ref{sec:pre}.
In particular, for the remainder of this section we write
\begin{equation*}
f(z,t) \sim_z g(z,t)
\end{equation*}
for two functions
\begin{equation*}
f(z,t) = \sum_{n\in\Z} a_t(n) z^n,\qquad g(z,t) = \sum_{n\in\Z} b_t(n) z^n,
\end{equation*}
if for all $n$
\begin{equation*}
a_t(n) \sim b_t(n) \quad \left({\rm as } \; t\to 0^+\right).
\end{equation*}
Using \eqref{Bernid} (noting that $\frac{4 \pi^2 y}{t} < 2 \pi$), we find that for $z = iy$,
\begin{equation}\label{asge}
g_\ell\left(z;\frac{it}{2\pi}\right)\sim_z\frac{1}{(-2\pi z)^\ell}\exp\left(-\ell S\left(\frac{4\pi^2 z}{t}\right)\right) = \frac{1}{(-2\pi z)^\ell} \frac{\left(\frac{4\pi^2 z}{t}\right)^{\ell}e^{\frac{2\pi^2 \ell z}{t}}}{\left(e^{\frac{4\pi ^2 z}{t}}-1\right)^{\ell}}.
\end{equation}

We also require that $D_{-j}$ is a quasimodular form of weight $\ell-j$ and thus satisfies (see \eqref{Eistran})
\begin{equation}\label{Dquasi}
D_{-j}\left(\frac{it}{2\pi}\right)\ll t^{j-\ell}.
\end{equation}

We now distinguish 2 cases, depending on whether $\ell$ is even or odd.
\subsection{Proof of Theorem \ref{main} for $\ell$ odd}
We first assume that $\ell$ is odd.
Plugging the definition \eqref{defineF} of $\mathcal{F}_j$ in \eqref{hform} gives
\begin{equation}
\label{E:hellodd}
H_{s+\frac{\ell}{2}}\left(\frac{it}{2\pi}\right)=-\sum_{j=0}^{\frac{\ell-1}{2}}\frac{D_{-2j-1}\left(\frac{it}{2\pi}\right)}{(2j)!}(2\ell)^jt^{-j}\mathcal F_{j,\frac12-\frac{s}{\ell}}(2\ell t).
\end{equation}
Now, by \eqref{Dquasi} and \eqref{Fas}, we obtain that
\[
D_{-2j-1}\left(\frac{it}{2\pi}\right)\mathcal{F}_{j,\frac12-\frac{s}{\ell}}(2 \ell t) t^{-j}\ll t^{-\ell+2j+1}.
\]
Thus (if we show that this is not vanishing) the dominant term in \eqref{E:hellodd} comes from $j=0$ and gives, by \eqref{Fas2},
\begin{equation}
\label{E:maintermj=0}
H_{s+\frac{\ell}{2}}\left(\frac{it}{2\pi}\right)\sim - \frac12D_{-1}\left(\frac{it}{2\pi}\right).
\end{equation}

To determine the asymptotic behavior of $D_{-1}$, we compute, using \eqref{asge} and \eqref{fullF},
\begin{align*}
D_{-1}\left(\frac{it}{2\pi}\right) &= 2\pi i \operatorname{coeff}_{z^{-1}} g_{\ell}\left(z;\frac{it}{2\pi }\right) \sim 2\pi i\left(-\frac{2\pi}{t}\right)^\ell {\rm coeff}_{z^{-1}} \frac{e^{\frac{2\pi^2 \ell z}{t}}}{\left(e^{\frac{4\pi ^2 z}{t}}-1\right)^\ell}\\
&= 2\pi i\left(-\frac{2\pi }{t}\right)^\ell \frac{t}{4\pi^2} {\rm coeff}_{z^{-1}} \frac{e^{\frac{\ell z }{2}}}{\left(e^z-1\right)^\ell}=\frac{-i\left(\frac{2\pi}{t}\right)^{\ell-1}}{(\ell-1)!}B_{\ell-1}^{(\ell)}\left(\frac{\ell}{2}\right).
\end{align*}
Thus when $\ell$ is odd, we have $\mathcal{C}_\ell = (-1)^{\frac{\ell-1}{2}}\frac{B_{\ell-1}^{(\ell)}(\frac{\ell}{2})}{2(\ell-1)!}$.

The theorem statement then follows from \eqref{asFh} and Proposition \ref{Berpolpart}.

\qed

\subsection{Proof of Theorem \ref{main} for $\ell$ even.}
We next assume that $\ell$ is even. Plugging \eqref{defineF} in \eqref{hform} gives
\begin{equation*}
H_{s+\frac{\ell}{2}}\left(\frac{it}{2\pi}\right) =\sum_{1\leq j \leq \frac{\ell}{2}} \frac{D_{-2j}\left(\frac{it}{2\pi}\right)}{(2j-1)!}  \left(\frac{t}{2\ell}\right)^{\frac12-j} \mathcal G_{j,\frac12-\frac{s}{\ell}}\left(\frac{\ell t}{2}\right).
\end{equation*}
We have, by \eqref{Gas},
\begin{equation*}
\mathcal G_{j,\frac12-\frac{s}{\ell}}\left(\frac{\ell t}{2}\right)\sim \frac{(j-1)!}{\sqrt{2\ell t}}.
\end{equation*}
Now, by \eqref{Dquasi},
\begin{equation*}
D_{-2j}\left(\frac{it}{2\pi}\right) \mathcal G_{j,\frac12-\frac{s}{\ell}}\left(\frac{\ell t}{2}\right) t^{\frac12-j} \ll  t^{-\ell +j}.
\end{equation*}
Thus the dominant term comes from $j=1$ and gives asymptotically (if non-vanishing)
\begin{equation}
\label{E:D-2}
\frac{1}{t} D_{-2}\left(\frac{it}{2\pi}\right).
\end{equation}

Now, by \eqref{fullF} and \eqref{asge}, we obtain
\begin{align*}
D_{-2}\left(\frac{it}{2\pi}\right) &= (2\pi i)^2 \coeff_{z^{-2}} g_\ell\left(z;\frac{it}{2\pi}\right) \sim (2\pi i)^2 \left(\frac{2\pi}{t}\right)^\ell \left(\frac{t}{4\pi ^2}\right)^2 \coeff_{z^{-2}}\frac{e^{\frac{\ell z}{2}}}{\left(e^z-1\right)^\ell}\\
&=- \left(\frac{2\pi}{t}\right)^{\ell-2} \frac{B_{\ell-2}^{(\ell)}\left(\frac{\ell}{2}\right)}{(\ell-2)!}.
\end{align*}
Thus when $\ell$ is even, $\mathcal{C}_\ell = (-1)^{\frac{\ell}{2}+1}\frac{1}{2\pi}\frac{B_{\ell-2}^{(\ell)}(\frac{\ell}{2})}{(\ell-2)!}$.

Thus the claim follows from \eqref{E:D-2} and Proposition \ref{Berpolpart}.
\qed

\subsection{Proof of Corollary \ref{cor:intro}}
The asymptotic formula \eqref{asVs} follows directly from Theorem \ref{main} by plugging in \eqref{constantcoeff}, combined with \eqref{etaas}.

For the second claim, we recall that
\begin{equation*}
{\rm qdim}(L_{\frak{g}}(-\Lambda_0 (s+1) +s \Lambda_1))=\lim_{t \to 0^+} \frac{  {\rm ch}\big[L_{\frak{g}}(-\Lambda_0 (s+1) +s \Lambda_1)\big](it)}{ {\rm ch}\big[L_{\frak{g}}(-\Lambda_0 )\big](it)}
\end{equation*}
and observe that \eqref{asVs} is independent of $s$.
\qed



\begin{remarks}\
\\
(1) In \cite{AP}, the following fusion rules formula
$$V_{s_1} \boxtimes V_{s_2}=V_{s_1+s_2},$$
was established, where, for convenience, we let $V_{s}:=L_{\frak{g}}(-\Lambda_0 (s+1) +s \Lambda_1))$, for $s \geq 0$,  and $V_{-s}:=L_{\frak{g}}(-\Lambda_0 (-s+1) -s \Lambda_{\ell-1}))$ for $s<0$. In this formula the symbol $\boxtimes$ denotes the {\em fusion product}, that is, it indicates that the space of intertwining operators for
this particular triple of modules is one-dimensional and zero otherwise. In other
words the fusion ring generated by isomorphism classes of ordinary irreducible $L_{\frak{g}}(-\Lambda_0)$ modules is isomorphic to $\mathbb{C}[\mathbb{Z}]$.
Note that Corollary \ref{cor:intro} is fully in agreement with this formula as quantum dimensions give the trivial one-dimensional representation of the fusion ring.
\\
(2) In \cite{BM2014}, the first and the third author studied asymptotic properties of the Fourier coefficients of the Bloch-Okounkov
 $n$-point functions. They also discussed level\footnote{Not to be confused with the level appearing in this paper.} $\ell$ $\in \mathbb{N}$ $n$-point functions $F^{(\ell)}$. As shown in  \cite[Section 4]{BM2014}, the Bloch-Okounkov $1$-point function  of level $\ell$ is given by  $F^{(\ell)}(z;\tau)=\frac{1}{\Theta(z;\tau)^{\ell}}$, where $\eta(\tau)^3 \Theta(z;\tau)= \vartheta(z;\tau)$. Theorem \ref{main} can be used to derive the leading asymptotics of its Fourier coefficients. After slight adjustments, due to additional Euler factors, we immediately get that the $s$-th Fourier coefficient $F^{(\ell)}_{s}$ satisfies (as $t \to 0^+$)
$$F^{(\ell)}_{s}\left(\frac{i t}{2 \pi}\right) \sim \mathcal{C}_{\ell},$$
so it has no growing term.
\end{remarks}

\section{A full asymptotic expansion for $\ell=3$}\label{sec:caseell3}

In this section we specialize to $\ell=3$ (i.e., $\frak{g}=\frak{sl}_3$) and explicitly work out the full asymptotic expansion in this case.


\begin{proposition}
\label{P:ell=3Asymp}
For  $\Lambda \in \{\Lambda_1,\Lambda_2 \}$ we have, as $t \to 0^+$,
	\begin{align*}
	&{\rm ch}\big[ L_{\frak{g}}(-(s+1)\Lambda_0+s \Lambda)\big]\left(\frac{it}{2\pi}\right) \\
	&\qquad=  e^{-\frac{5\pi^2}{6t}}\left(\frac{2\pi}{t}\right)^{\frac52}\left(
\left(\frac{\pi^2}{4t^2}+\frac{3}{4t}\right)\sum_{n=0}^N \frac{E_{2n}\left(\frac12-\frac{s}{3}\right)}{n!}\left(-\frac{3t}{2}\right)^n  \right. \\
& \qquad \qquad \qquad \qquad \qquad + \left. \frac{3}{8t} \sum_{n=0}^N \frac{E_{2n+2}\left(\frac12-\frac{s}{3}\right)}{n!} \left(-\frac{3t}{2}\right)^n + O\left(t^{N+1}\right)  \right).
	\end{align*}
In particular, for $s=0$ we have
\begin{multline} \label{s=0}
	{\rm ch}\big[ L_{\frak{g}}(- \Lambda_0)\big]\left(\frac{it}{2\pi}\right) \\= e^{-\frac{5\pi^2}{6t}}\left(\frac{2\pi}{t}\right)^{\frac52} \left(\frac{\pi^2}{4 t^2}+\frac{3}{4t}-\frac{1}{4t} \frac{\partial}{\partial t} \right)\left(\sum_{n=0}^N \frac{E_{2n}}{n!} \left(\frac{-3t}{8}\right)^n + O\left(t^{N+1}\right) \right).
\end{multline}

\end{proposition}
\begin{proof}
We have, by \eqref{E:F=h},
\begin{equation}
\label{E:F3s}
F_{3,s}(q) = i(q)_\infty^3 q^{-h_s-\frac38} H_{s+\frac32}(\tau).
\end{equation}
Now, by \eqref{E:hellodd}
\begin{equation}\label{expandH}
H_{s+\frac32}\left(\frac{it}{2\pi}\right) = -D_{-1}\left(\frac{it}{2\pi}\right)\mathcal F_{0,\frac12-\frac{s}{3}}(6 t) - D_{-3}\left(\frac{it}{2\pi}\right)\frac{3}{t} \mathcal F_{1,\frac12-\frac{s}{3}}(6 t).
\end{equation}
We compute the Laurent coefficients $D_{-j}$, using \eqref{thetatay} and \eqref{E:gDdef},
\begin{equation*}
D_{-3}(\tau) = i,\qquad D_{-1}(\tau) = -\frac{3i}{8 \pi^2} G_2(\tau).
\end{equation*}
Using
\begin{align*}
G_2\left(\frac{it}{2\pi}\right) = -\frac{4\pi^4}{3t^2} + \frac{4\pi^2}{t} +O\left(t^{-2}e^{-\frac{4\pi^2}{t}}\right)
\end{align*}
\eqref{fullF2}, and \eqref{expandH}, we then obtain
\begin{align}\notag
H_{s+\frac32}\left(\frac{it}{2\pi}\right)&=\left(-\frac{\pi^2 i}{4t^2}-\frac{3i}{4t}\right)\sum_{n=0}^N \frac{E_{2n}\left(\frac12-\frac{s}{3}\right)}{n!} \left(-\frac{3t}{2}\right)^n\\
&\quad - \frac{3i}{8t} \sum_{n=0}^N \frac{E_{2n+2}\left(\frac12-\frac{s}{3}\right)}{n!} \left(-\frac{3t}{2}\right)^n+O\left(t^{N+1}\right).\label{has}
\end{align}
Here we employ \eqref{EB} with $m=2$.

To finish the claim, we plug \eqref{E:F3s} into \eqref{constantcoeff}, to obtain that
\begin{align*}
{\rm ch}\big[ L_{\frak{g}}(-(s+1)\Lambda_0+s \Lambda_1)\big](\tau)
= \frac{i}{\eta(\tau)^5} H_{s+\frac32}(\tau).
\end{align*}
Using \eqref{etaas} and \eqref{has}, then yields the claims. Formula \eqref{s=0} follows easily by \eqref{Euler12}.
\end{proof}

Proposition \ref{P:ell=3Asymp} immediately leads to the following asymptotic expansion for the quantum dimension.
\begin{corollary} \label{C:Asym}
For $\Lambda \in \{ \Lambda_1, \Lambda_2 \}$ we have, as $t \to 0^+$,
\begin{equation} \label{quotient}
 \frac{{\rm ch}\big[ L_{\frak{g}}(-(s+1)\Lambda_0+s \Lambda)\big]\left(it \right)}{{\rm ch}\big[ L_{\frak{g}}(- \Lambda_0)\big]\left(it \right)} \sim 1- \frac{s^2(\pi^2-1)}{3 \pi}t +O\left(t^2\right).
 \end{equation}

\end{corollary}

\begin{remark} (Heisenberg patterns) Let $M(1)$ be the rank one Heisenberg vertex operator algebra.
For an $M(1)$-module $M(1,s)$, we easily get
\begin{equation}
\label{Heisen}
\frac{{\rm ch}[M(1,s)]\left(it \right)}{{\rm ch}\big[M(1)]\left(it \right)}=e^{-2 \pi s^2 t}=1-2 \pi s^2 t + O\left(t^2\right).
\end{equation}
The reader should notice the similarity between \eqref{Heisen} and Corollary \ref{C:Asym} as they both have an $s^2$ in the second term of the asymptotic expansion (incidentally
both vertex algebras share ``additive'' fusion rules \cite{AP,AP2}). It would be very interesting to find a closed expression for the asymptotic expansion in \eqref{quotient}.
\end{remark}

\section{A decomposition for $F_{\ell,s}(\zeta_1,\ldots,\zeta_{\ell-1})$}
\label{sec:decomp}

In this part we prove Theorem \ref{T:Fls} from the introduction.

\begin{proof}[Proof of Theorem \ref{T:Fls}]
	We write
	\begin{equation}
	\label{E:F=calF}
	F_{\ell,s}(\zeta_1,\ldots,\zeta_{\ell-1})=(-i)^\ell q^{\frac{\ell}{8}}\prod_{j=1}^{\ell-1} \zeta_j^{-\frac{j}{2}}(q)^{\ell+1}_{\infty}\mathcal F_{\ell,s+\frac{\ell}{2}}.
	\end{equation}
	Here, for $r\in\mathbb Z +\frac{\ell}{2}$,
	\begin{equation*}
	\mathcal F_{\ell,r}:=\text{coeff}_{\zeta_\ell^r}\mathcal F_\ell(z_\ell),
	\end{equation*}
	where
	\[
	\mathcal F_{\ell}(z_\ell)=\mathcal F_{\ell}(z_1,\dots,z_{\ell-1},z_{\ell}):=\frac{(-1)^\ell}{\prod_{j=1}^{\ell}\vartheta(w_j)}.
	\]
	Using \eqref{thetatran}, we have
	\[
	\mathcal F_\ell(z_\ell +1)=(-1)^\ell\mathcal F_\ell(z_\ell),\qquad \mathcal F_\ell(z_\ell -\tau)=(-1)^\ell q^{\frac{\ell}{2}}\mathcal F_\ell(z_\ell) \prod_{j=1}^{\ell}\zeta_j^{-j} .
	\]
	We now find a formula for $\mathcal{F}_{\ell,r}$ that is closely related to Theorem \ref{hltheorem} above (which was proven in \cite{BRZ} by a similar argument). In particular, we compute the following integral in two ways:
	\begin{equation}\label{computeint}
	\int_{\partial P_{z_0}} \mathcal F_\ell(w) \vartheta_{r,\varepsilon,\frac{\ell}{2}}^+ \left(w+\frac{1}{\ell} \sum_{j=1}^{\ell-1} jz_j\right) dw.
	\end{equation}
	Here $z_0$ is chosen such that $w_j \in P_{z_0}$ for all $j$ (which is possible due to the restrictions on the $\zeta_j$); note that $-v < {\rm Im}(z_0) < 0$.
	
	For the first computation, we also use the elliptic shifts (as above, $\varepsilon \equiv \ell \pmod{2}$)
	\begin{align*} &(-1)^\ell q^{\frac{\ell}{2}} e^{-2 \pi i \ell z} \prod_{j = 1}^{\ell - 1} \zeta_j^{-j} \vartheta_{r,\varepsilon,\frac{\ell}{2}}^+\left(z+\frac{1}{\ell}\sum_{j=1}^{\ell-1}jz_j -\tau\right)
	 -\vartheta_{r,\varepsilon,\frac{\ell}{2}}^+\left(z+\frac{1}{\ell}\sum_{j=1}^{\ell-1}jz_j\right)\\
	 &\phantom{\vartheta_{r,\varepsilon,\frac{\ell}{2}}^+\left(z+\frac{1}{\ell}\sum_{j=1}^{\ell-1}jz_j\right)}
	=(-1)^\ell q^{\frac{(r+\ell)^2}{2\ell}}e^{-2\pi i(r+\ell)z}\prod_{j=1}^{\ell-1}\zeta_j^{-\frac{j(r+\ell)}{\ell}},\\
	 &\vartheta_{r,\varepsilon,\frac{\ell}{2}}^+\left(z+\frac{1}{\ell}\sum_{j=1}^{\ell-1}jz_j\right) = (-1)^\ell \vartheta_{r,\varepsilon,\frac{\ell}{2}}^+\left(z+\frac{1}{\ell}\sum_{j=1}^{\ell-1}jz_j+1\right),
	\end{align*}
	to obtain  that \eqref{computeint} equals
	\begin{equation}
	\label{E:Fellrfirst}
	(-1)^\ell q^{\frac{(r+\ell)^2}{2\ell}}  \mathcal{F}_{\ell, r + \ell}\prod_{j = 1}^{\ell - 1} \zeta_j^{-\frac{(r+\ell)j}{\ell}}.
	\end{equation}
	
	On the other hand, the Residue Theorem implies that \eqref{computeint} also equals
	\[
	2\pi i \sum_{w\in S_{z_0}(\tau)}\text{Res}_{z=w}\mathcal F_\ell(z)\vartheta_{r,\varepsilon,\frac{\ell}{2}}^+\left(z+\frac{1}{\ell}\sum_{j=1}^{\ell-1}jz_j\right).
	\]
	Now, for $1 \leq j \leq \ell$, $\mathcal F_\ell$ has simple poles at $w_j$. We compute
	\begin{align*}
	\text{Res}_{w=w_\nu}\mathcal F_\ell(w) = -\frac{1}{2\pi \eta^3}\frac{1}{\prod_{\substack{j=1\\j\neq \nu}}^{\ell}\vartheta(w_\nu-w_j)}.
	\end{align*}
	Comparing to \eqref{E:Fellrfirst}, we find that
	\begin{equation*}
	\mathcal{F}_{\ell, r+\ell} = -i (-1)^\ell q^{-\frac{(r+\ell)^2}{2\ell}} \prod_{j =1}^{\ell - 1} \zeta_j^{\frac{(r+\ell)j}{\ell}} \eta(\tau)^{-3} \sum_{\nu=1}^{\ell}\frac{\vartheta_{r,\ell,\frac{\ell}{2}}^+\left(w_\nu-\frac{1}{\ell}\sum_{j=1}^{\ell}w_j \right)}
	{\prod_{\substack{j=1\\j\neq\nu}}^{\ell} \vartheta\left(w_\nu-w_j\right)}.
	\end{equation*}
	Plugging in to \eqref{E:F=calF} and setting $r = s - \frac{\ell}{2}$ gives the claim.
	
\end{proof}

\begin{remarks}\ \\
	(1)\ One could prove related decompositions without the condition that the $w_j$ are distinct. Any repeated values among the $w_j$ require modified calculations of residues in the proof above, leading to partial derivatives of $\vartheta^+$ (as in Theorem \ref{T:F=Dj}, which is related to the limiting case that all $w_j$ have the repeated value $0$). \\
	(2)\ Theorem \ref{T:Fls} could be used to determine the asymptotic behavior as $t\to 0^+$ of $F_{\ell,s}(\zeta_1,\ldots,\zeta_{\ell-1})$, so long as the $\zeta_j$ are also bounded in the ranges given in the theorem.  \\
	(3)\ We stress that Theorem \ref{T:Fls} also applies to $\ell=2$, which is quite different compared to $\ell \geq 3$ because the
	relevant vertex algebra is not an affine Lie algebra). Furthermore, note that in the case $\ell = 1$, Theorem \ref{T:Fls} is equivalent to Theorem \ref{T:F=Dj}.
\end{remarks}

\section{Modular-type transformation properties of $F_{\ell,s}$}

In this section we discuss general modular transformation properties of $F_{\ell,s}$.
The modular transformation properties of the denominator appearing in the formula for $F_{\ell,s}$ are well-understood, so we only have to consider the numerator.
We note that in the parlance of \cite{CM1,CMW}, $\vartheta^+_{s + \frac{\ell}{2}, \ell, \frac{\ell}{2}}(z; \tau)$ is an example of a regularized (Jacobi) false theta function. The following proposition gives the modular transformation properties of these functions.

\begin{proposition} \label{false-mod-general}
For $z \in \mathbb{C}\backslash\mathbb R$ and $\left(\begin{smallmatrix}
a & b \\ c & d
\end{smallmatrix}\right)\in\SL_2(\Z)$ with $c\neq 0$, we have
	\begin{equation*}
	\begin{aligned}	 &\vartheta^+_{r,\varepsilon,M}\left(z;\frac{a\tau+b}{c\tau+d}\right)\\
	&= \sqrt{-\frac{i(c\tau+d)}{2}}\sum_{j=0}^{2c-1}(-1)^{j\varepsilon}\zeta^{a(2Mj-r)^2}_{4cM}\Bigg(\zeta^{-r+2Mj}\int_{\mathbb R}\frac{e^{\frac{\pi i(c\tau+d)x^2}{2}-\frac{\pi i}{\sqrt{cM}}(r-2Mj)x}}{1-\zeta^{4cM}e^{4\pi i\sqrt{cM}x}}dx\\
	&\hspace{2.5cm}+\frac{1}{4\sqrt{cM}}(1-{\rm sgn}(\im(z)))\zeta_{8cM}^{2Mj-r} e^{2\pi i cM (c\tau+d)\left(z-\frac{1}{8cM}\right)^2}\\
	&\hspace{3cm}\times  \vartheta\left(\frac{c\tau + d}{2}\left(z-\frac{1}{8cM}\right) - \frac12+\frac{r-2Mj}{4cM}; \frac{c\tau+d}{8cM} \right)\Bigg).
	\end{aligned}
	\end{equation*}

\end{proposition}

\begin{proof}
	We assume without loss of generality that $c>0$. Plugging in the series representation of $\vartheta^+$ and writing $\frac{a\tau+b}{c\tau+d}=\frac{a}{c}-\frac{1}{c(c\tau+d)}$, we find that
	\begin{align*}
	\vartheta_{r,\varepsilon,M}^+\left(z;\frac{a\tau+b}{c\tau+d}\right) = \sum_{j=0}^{2c-1}(-1)^{j\varepsilon} \zeta_{4cM}^{a(2Mj-r)^2}  \vartheta_{r-2Mj,0,2cM}^+\left(z;-\frac{2}{c\tau+d}\right).
	\end{align*}

We next invert each term in the sum by finding an inversion formula for $\vartheta_{r, 0,M}^+(z;\frac{i}{w})$. If $\im(z)>0$, then we use
\begin{equation*}
e^{-\frac{\pi i }{\tau} A^2}=\sqrt{- i \tau} \int_{\mathbb{R}} e^{\pi i \tau x^2 + 2 \pi i A x} dx,
\end{equation*}
to obtain that
\begin{align}
\label{E:thetai/w-general}
\vartheta_{r,0,M}^+\left(z;\frac{i}{w}\right) = \sqrt{w} \zeta^{-r} \int_\R \frac{e^{-\pi wx^2-\frac{2\pi ir}{\sqrt{2M}}x}}{1-\zeta^{2M}e^{2\pi i\sqrt{2M}x}}dx.
\end{align}
Plugging into \eqref{E:thetai/w-general} gives the claim in this case.

If $\im(z)<0$, then we complete the theta function, calculating
\begin{align*}
&\vartheta_{r,0,M}^+(z;\tau) + \vartheta^+_{-r-2M,0,M}(-z;\tau)= q^{\frac{(M-r)^2}{4M}}  \zeta^{M-r} \vartheta\left(2Mz-\frac{1}{2}+(M-r)\tau;2M\tau\right).
\end{align*}
For the partial theta function $\vartheta^+_{-r-2M,0,M}(-z;\tau)$ we again use \eqref{E:thetai/w-general}. which is invariant under $r \mapsto -r - 2M, z \mapsto -z,$ and multiplication by $-1$. From \eqref{etatran} we obtain for the theta function
\begin{multline*}
e^{-\frac{\pi(M-r)^2}{2Mw}}  \zeta^{M-r} \vartheta\left(2Mz-\frac{1}{2}+(M-r)\frac{i}{w};-\frac{2M}{iw}\right)\\
= -i\zeta_{4M}^{M-r}e^{-2\pi Mw\left(z-\frac{1}{4M}\right)^2}\sqrt{\frac{w}{2M}} \vartheta\left(izw-\frac{iw}{4M}-\frac12+\frac{r}{2M};\frac{iw}{2M}\right).
\end{multline*}
Plugging in gives the claim.
\end{proof}

As a special case, we recover Proposition 7 of \cite{CMW}.
\begin{corollary}
\label{C:thetainv-general}
 For $z \in\mathbb C\backslash \mathbb{R}$, we have\\
\textnormal{(i)} If $\ell$ is odd, then we have
 \begin{multline*}(-i\tau)^{-\frac12}\vartheta_{s+\frac{\ell}{2},1,\frac{\ell}{2}}^+ \left(z;-\frac{1}{\tau}\right)\\= \frac12\int_{\mathbb{R}} \frac{e^{\pi i\tau x^2} e^{ \frac{2 \pi i}{\sqrt{\ell}} (s+\ell)\left(x-\sqrt{ \ell} z\right)}}{\cos\left(\sqrt{\ell} \pi \left(x- \sqrt{\ell} z\right)\right)} dx
	+\frac{\left(1-\sgn(\Im(z)\right)}{2\sqrt{\ell}} e^{\pi i\ell \tau z^2}  \vartheta\left(\tau z+\frac{s}{\ell};\frac{\tau}{\ell}\right).
	\end{multline*}
\textnormal{(ii)} If $\ell$ is even, then we have
\begin{align*}
(-i\tau)^{-\frac12}\vartheta_{s+\frac{\ell}{2},0,\frac{\ell}{2}}^+ \left(z;-\frac{1}{\tau}\right)&=-\frac{i}{2} \int_{\mathbb{R}} \frac{e^{\pi i\tau x^2} e^{ \frac{2\pi i}{\sqrt{\ell}} (s+\ell)\left(x-\sqrt{ \ell} z\right)}}{\sin\left(\sqrt{\ell} \pi \left(x- \sqrt{\ell} z\right)\right)} dx \\
&+\frac{\left(1-\sgn(\Im(z)\right)}{2i\sqrt{\ell}} e^{-\frac{\pi i s}{\ell}} e^{\pi i \ell \tau \left(z - \frac{1}{2\ell}\right)^2}  \vartheta\left(\tau z+\frac{s}{\ell}-\frac{\tau}{2\ell};\frac{\tau}{\ell}\right).
\end{align*}
\end{corollary}

\begin{proof}
The claim follows from Proposition \ref{false-mod-general} with $a=d=0$, $c=-b=1$, using the easily verified identity
	\begin{equation*}
	\frac12 e^{-\frac{\pi iz}{2}-\frac{\pi i}{4}+\frac{\pi i\tau}{16}}\left(\vartheta\left(\frac{z}{2}-\frac{\tau}{8}-\frac14;\frac{\tau}{4}\right)+i\vartheta\left(\frac{z}{2}-\frac{\tau}{8}+\frac14;\frac{\tau}{4}\right)\right) = \vartheta(z;\tau).
	\end{equation*}
\end{proof}

\begin{remark}
	In view of Theorem \ref{T:Fls}, the modular transformation formulas in this section can be now used to compute modular transformation formulas for
	${\rm ch}[M](z;\tau)$ for every irreducible module $M$. This result then can be interpreted as a statement on characters of  more general
	class of irreducible $L_{\frak{g}}(-\Lambda_0)$-modules parametrized by a continuous parameter.
	Presumably, such results can be used to formulate a Verlinde type formula  for $L_{\frak{g}}(-\Lambda_0)$-modules.
	As this analysis is clearly beyond the scope of this paper, it is left for future investigation.
\end{remark}

\appendix

\rm

\section{Coefficients of Bernoulli polynomials}

Here we calculate some special values for certain generalized Bernoulli polynomials. In particular, we are interested in the values
\begin{equation*}
\mathcal C_{\ell}^\ast := \begin{cases}
(-1)^{\frac{\ell-1}{2}} \frac{B_{\ell-1}^{(\ell)}\left(\frac{\ell}{2}\right)}{2(\ell-1)!}\quad &\textnormal{ if $\ell$ is odd},\vspace{4pt}\\
(-1)^{\frac{\ell}{2}+1}\frac{1}{2\pi} \frac{B_{\ell-2}^{(\ell)}\left(\frac{\ell}{2}\right)}{(\ell-2)!} \quad&\textnormal{ if $\ell$ is even}.
\end{cases}
\end{equation*}

\begin{proposition}\label{Berpolpart}
	We have
	\[
\mathcal C_\ell = \mathcal C_\ell^\ast.
\]
\end{proposition}
\begin{proof}
	We offer 2 proofs. The first proof uses recurrences. We first assume that $\ell$ is odd. It is direct that
	\begin{equation}\label{rec}
	\mathcal C_1=\mathcal C_1^*=\frac12,\quad \mathcal C_{\ell+2}=\frac{\ell}{4(\ell+1)}\mathcal C_\ell.
	\end{equation}
	Thus the identity follows, once we show that $\mathcal C_{\ell}^*$ also satisfies \eqref{rec} (with $\mathcal C_\ell$ replaced by $\mathcal C_\ell^*$). For this we use the following integral representation of $\mathcal C_\ell^*$
	\[
	 B_{\ell-1}^{(\ell)}\left(\frac{\ell}{2}\right)=-2^{2-\ell}(\ell-1)!\frac{i^{-\ell}}{2\pi i}\int_{\mathbb R}\frac{1}{\cosh(z)^\ell}dz
	\]
	which follows by a direct residue calculation. The recurrence \eqref{rec} now follows, using
	\begin{equation}\label{diffcos}	
	 \frac{d^2}{dz^2}\frac{1}{\cosh(z)^\ell}=\frac{\ell^2}{\cosh(z)^\ell}-\frac{\ell(\ell+1)}{\cosh(z)^{\ell+2}}.
	\end{equation}
	
	We next assume that $\ell$ is even. Then
	\begin{equation*}
	\mathcal C_2=\mathcal C_2^* = \frac{1}{2 \pi},\qquad \mathcal C_{\ell+2} = \frac{\ell}{4(\ell+1)} \mathcal C_\ell.
	\end{equation*}
	In this case, we have the representation
	\begin{equation*}
	B_{\ell-2}^{(\ell)}\left(\frac{\ell}{2}\right) = -2^{3-\ell} (\ell-2)! i^{-\ell} \int_{\R} \frac{z}{\cosh(z)^\ell}dz.
	\end{equation*}
	Using \eqref{diffcos}, the integral becomes
	\begin{equation*}
	\frac{\ell}{\ell+1} \int_\R \frac{z}{\cosh(z)^{\ell+2}}dz - \frac{1}{\ell(\ell+1)} \int_\R z \frac{d^2}{dz^2} \frac{1}{\cosh(z)^\ell} dz.
	\end{equation*}
	Using integration by parts on the second term then gives the claim. \\
\noindent {\em Second proof.} The second proof uses the formal Residue Theorem. Writing  $w=g(z) \in z\mathbb{C}[[z]]$ with $g'(0) \neq 0$, we have
$${\rm Res}_{w=0} f(w)={\rm Res}_{z=0} f(g(z)) g'(z).$$

For $\ell$ odd, we
let $f(w):=\frac{(w+1)^{\frac{\ell}{2}-1}}{w^{\ell}}$ and $w:=g(z)=e^{z}-1$. Then
$${\rm Res}_{z=0} \frac{e^{\frac{\ell}{2} z}}{(e^z-1)^\ell}= {\rm Res}_{w=0} \frac{(w+1)^{\frac{\ell}{2}-1}}{w^{\ell}}  ={ \frac{\ell}{2}-1 \choose \ell -1}.$$
This implies the statement.

For $\ell$ even, we have
$$I_{\ell}:={\rm Res}_{w=0}  \frac{\log(1+w) (w+1)^{\frac{\ell}{2}-1}}{w^{\ell}}={\rm Res}_{z=0} \frac{z e^{\frac{\ell}{2} z}}{(e^z-1)^\ell}.$$
Expanding $\log(1+w)=\sum_{n=1}^\infty (-1)^{n+1} \frac{w^n}{n}$, we get
\begin{equation*}
I_{\ell}={\rm Res}_{w=0} \frac{(w+1)^{\frac{\ell}{2}-1}}{w^\ell} \sum_{n=1}^\infty (-1)^{n+1} \frac{w^n}{n} = \sum_{j=0}^{\frac{\ell}{2}-1} {\frac{\ell}{2}-1 \choose j} \frac{(-1)^{\frac{\ell}{2}+1+j}}{\frac{\ell}{2}+j}.
\end{equation*}
To this end, we need the evaluation
\begin{equation} \label{identity}
\sum_{k=0}^n {n \choose k} (-1)^k \frac{1}{k+c}=\frac{n! (c-1)!}{(n+c)!}.
\end{equation}
To see \eqref{identity}, we integrate the Binomial Theorem to obtain (here ${\rm B}(a,b)$ is the usual Beta function)
\begin{align*}
\sum_{k=0}^n \binom{n}{k} \frac{(-1)^k}{k+c}&=\int_0^1 \sum_{k=0}^n \binom{n}{k} (-1)^k x^{k+c-1} dx = \int_0^1 x^{c-1} (1-x)^n dx ={\rm B}(c,n+1) \\
&= \frac{\Gamma(c) \Gamma(n+1)}{\Gamma(c+n+1)},
\end{align*}
by the standard evaluation of the Beta integral.

Letting $n=\frac{\ell}{2}-1$ and $c=\frac{\ell}{2}$ in (\ref{identity}), we then obtain $I_\ell=(-1)^{\frac{\ell}{2}+1} \frac{(\frac{\ell}{2}-1)! (\frac{\ell}{2}-1)!}{(\ell-1)!}$ as desired.

\end{proof}

\end{document}